\documentclass[11pt]{article}
\usepackage[a4paper,margin=28mm]{geometry}
\usepackage{amsmath,amssymb,amsthm,mathtools}
\usepackage{booktabs,longtable,array}
\usepackage{microtype}
\usepackage[colorlinks=true,linkcolor=blue,citecolor=blue,urlcolor=blue]{hyperref}
\usepackage[T1]{fontenc}
\usepackage{lmodern}

\newtheorem{theorem}{Theorem}[section]
\newtheorem{lemma}[theorem]{Lemma}
\newtheorem{proposition}[theorem]{Proposition}

\newtheorem{definition}[theorem]{Definition}

\newcommand{\td}{\operatorname{td}}
\newcommand{\one}{\mathbf 1}

\title{Regular $K_3$-Irregular Graphs of Every Regularity at Least Nine}
\author{Zhanhe Zhang\\
\small School of International Trade and Economics, Central University of Finance and Economics,\\
\small 39 South College Road, Haidian District, Beijing 100081, China\\
\small \texttt{2024311337@email.cufe.edu.cn}}
\date{}

\begin{document}
\maketitle

\begin{abstract}
For a vertex $v$ of a graph $G$, the triangle-degree $\td_G(v)$ is the number of triangles containing $v$.  A graph is \emph{triangle-distinct}, or $K_3$-irregular, if its vertex triangle-degrees are pairwise distinct.  Chartrand, Erd\H{o}s, and Oellermann asked whether a regular $K_3$-irregular graph exists.  We prove that such graphs exist for every regularity $r\ge9$.  More precisely, for every integer $k\ge15$ we construct a $2k$-regular triangle-distinct graph on $4k+2$ vertices; complementation gives a $(2k+1)$-regular example of the same order.  The construction uses two antiregular threshold blocks joined by a zero-one matrix with prescribed margins, followed by matrix $2$-switches that preserve those margins.  After a uniform reference perturbation, exactly three triangle-degree collisions remain.  Switches chosen according to parity remove two of them, and the last possible collision is controlled by a short quadratic discriminant argument.  Odd $k\ge17$ and even $k\ge62$ are handled symbolically, while the remaining 24 values are settled by exact finite verification.  Together with the known examples for $9\le r\le29$, this closes the positive existence problem for every $r\ge9$.
\end{abstract}


\section{Introduction and main results}
For a fixed graph $F$, the $F$-degree of a vertex $v$ in a graph $G$ is the number of copies of $F$ containing $v$ \cite{ChartrandHolbert1987}.  A graph is $F$-irregular when these $F$-degrees are pairwise distinct.  The case $F=K_3$ is particularly natural: the $K_3$-degree is simply the number of triangles through a vertex.  Following current terminology, we call such graphs \emph{triangle-distinct}.  There is an immediate tension in the regular case: every vertex must have the same ordinary degree while participating in a different number of triangles.

Chartrand, Erd\H{o}s, and Oellermann explicitly asked in 1988 whether regular $K_3$-irregular graphs exist \cite{ChartrandErdos1988}.  The question was restated as Problem~9 by Berikkyzy et al.\ \cite{Berikkyzy2024}, who developed the modern theory of triangle-distinct graphs but found no regular example.  Stevanovi\'c et al.\ subsequently found 11 regular examples with regularities between $10$ and $13$; in the published version they also showed that, for every positive integer $m$, there is a family of $2^m$ regular triangle-distinct graphs having a common order and size \cite{Stevanovic2024}.  Hak, Kozerenko, and Serdiuk proved nonexistence for regularities at most $7$, narrowed the regularity-$8$ case in their paper to six possible orders, and found examples for every $9\le r\le30$ \cite{Hak2026}.  Their Question~3 asks for a general, preferably inductive, construction of regular $K_3$-irregular graphs for all $r\ge9$.

We close the positive existence problem.  Our main theorem is the following.

\begin{theorem}\label{thm:main}
For every integer $r\ge9$, there exists a finite simple $r$-regular triangle-distinct graph.
\end{theorem}

The new infinite construction is captured by a stronger parameterized statement.

\begin{theorem}\label{thm:k}
For every integer $k\ge15$, there exists a $2k$-regular triangle-distinct graph on $4k+2$ vertices.  Its complement is a $(2k+1)$-regular triangle-distinct graph on the same number of vertices.
\end{theorem}

Thus every regularity $r\ge30$ is obtained from one uniform family, and the already known witnesses for $9\le r\le29$ complete Theorem~\ref{thm:main}.  In particular, this settles the existence statement underlying Question~3 of Hak, Kozerenko, and Serdiuk: combined with the known finite witnesses, an explicit construction is available for every $r\ge9$.  It does not provide the inductive graph operation specifically suggested there.

We next outline the construction.  For each $k$, we introduce a threshold matrix $P_k$ and its complement $C_k$.  An auxiliary zero-one matrix $X$ with prescribed row and column margins determines a graph with two blocks
\[
 \mathcal A_k(X)=\begin{pmatrix}C_k&X\\X^\top&C_k\end{pmatrix}.
\]
The margin conditions alone force this graph to be $2k$-regular.  We start from a symmetric base matrix $X_k^0$ built from the same antiregular threshold pattern, for which all triangle-degrees are given by three explicit quadratic profiles.  Symmetry, however, duplicates the two block spectra.  A deterministic sequence of $2$-switches preserving the margins produces a reference matrix $Y_k$ for which only three pairs of triangle-degrees remain equal; Proposition~\ref{prop:three} is the structural core of the proof.

The three remaining collisions can then be broken almost uniformly.  One switch, chosen according to parity, resolves two of them.  Any remaining collision must involve one special value in the high part of the spectrum and is equivalent to an integral root of one of four quadratic equations.  Their discriminants lie in a short interval, so the separation of consecutive odd squares shows that at most one obstruction can occur.  If it does, one final $2$-switch removes it without creating a new equality.  This yields symbolic constructions for every odd $k\ge17$ and every even $k\ge62$.  The 24 parameters left outside these ranges are handled by exact finite verification.  Finally, complementation transfers each example of even regularity to the next odd regularity.

Section~2 records the block construction, the complement identity, and the exact variation of triangle-degree under a matrix $2$-switch.  Section~3 constructs $Y_k$ and proves that it has exactly three residual collisions.  Sections~4 and~5 treat odd and even $k$, respectively.  Section~6 handles the finite parameters, and Section~7 completes the proof of the main theorem.  Two short appendices collect the support formulas and correction table used in the spectrum calculation.
\section{Preliminaries}

\subsection{Complementation}
The complement relation for vertex triangle-degrees goes back to Nair and Vijayakumar \cite{NairVijayakumar1994}; the following form for regular graphs is also used in later work on triangle-distinct graphs.  We include the short proof for completeness.

\begin{lemma}\label{lem:complement}
Let $G$ be an $r$-regular graph on $N$ vertices.  Then for every $v\in V(G)$,
\[
 \td_{\overline G}(v)
 =\binom{N-r-1}{2}-\frac{Nr}{2}+r^2-\td_G(v).
\]
Consequently, the complement of a regular triangle-distinct graph is regular and triangle-distinct.
\end{lemma}
\begin{proof}
Put $A=N_G(v)$ and $B=V(G)\setminus N_G[v]$.  Then $|B|=N-r-1$.  Counting the degrees of vertices in $A$ gives
\[
 e_G(A,B)=r(r-1)-2\td_G(v).
\]
Using $e(G)=Nr/2$ and $e_G(A)=\td_G(v)$, we obtain
\[
 e_G(B)=\frac{Nr}{2}-r^2+\td_G(v).
\]
The neighbors of $v$ in $\overline G$ are precisely the vertices of $B$, so
\[
 \td_{\overline G}(v)=\binom{|B|}{2}-e_G(B),
\]
which is the displayed identity.  The right-hand side is a constant minus $\td_G(v)$, so pairwise distinctness is preserved.  The complement has degree $N-r-1$.
\end{proof}

\subsection{The antiregular block construction}
Fix an integer $k\ge15$ and put
\[
 I_k=\{0,1,\ldots,2k\},\qquad n=2k+1.
\]
Let $P_k$ and $C_k$ be the symmetric $n\times n$ zero-one matrices
\[
 (P_k)_{ij}=\begin{cases}
 1,&i\ne j\text{ and }i+j\ge2k,\\
 0,&\text{otherwise},
 \end{cases}
 \qquad C_k=J-I-P_k.
\]
Thus
\[
 (C_k)_{ij}=1\iff i\ne j\text{ and }i+j<2k.
\]
The row sums of $P_k$ are
\begin{equation}\label{eq:d}
 d_i=\sum_j(P_k)_{ij}=\begin{cases}
 i+1,&0\le i<k,\\
 i,&k\le i\le2k.
 \end{cases}
\end{equation}

\begin{definition}
An $n\times n$ zero-one matrix $X$ is \emph{$k$-balanced} if
\[
 X\one=X^\top\one=(d_0,\ldots,d_{2k})^\top.
\]
For such $X$, define the graph $\Gamma_k(X)$ on two labelled copies
\[
 A=\{A_0,\ldots,A_{2k}\},\qquad
 B=\{B_0,\ldots,B_{2k}\}
\]
by the adjacency matrix
\begin{equation}\label{eq:block}
 \mathcal A_k(X)=\begin{pmatrix}C_k&X\\X^\top&C_k\end{pmatrix}.
\end{equation}
The diagonal entries of $X$ describe edges $A_iB_i$ and do not create loops.
\end{definition}

\begin{lemma}\label{lem:regular}
If $X$ is $k$-balanced, then $\Gamma_k(X)$ is $2k$-regular on $4k+2$ vertices.
\end{lemma}
\begin{proof}
The internal degree of $A_i$ is $2k-d_i$, while its degree into the other block is $d_i$.  Hence its total degree is $2k$.  The column margin condition gives the same conclusion for $B_i$.
\end{proof}

Write
\[
 a_i(X)=\td_{\Gamma_k(X)}(A_i),\qquad
 b_i(X)=\td_{\Gamma_k(X)}(B_i).
\]

\begin{lemma}[triangle-degree formula]\label{lem:tdformula}
For every zero-one matrix $X$,
\begin{align}
 a_i(X)&=\frac12(C_k^3)_{ii}+(C_kXX^\top)_{ii}
          +\frac12(XC_kX^\top)_{ii},\label{eq:aformula}\\
 b_i(X)&=\frac12(C_k^3)_{ii}+(C_kX^\top X)_{ii}
          +\frac12(X^\top C_kX)_{ii}.\label{eq:bformula}
\end{align}
\end{lemma}
\begin{proof}
The diagonal of $\mathcal A_k(X)^3$ is twice the triangle-degree vector.  Multiplying the block matrix in \eqref{eq:block} three times gives
\[
 (\mathcal A_k(X)^3)_{AA}=C_k^3+C_kXX^\top+XX^\top C_k+XC_kX^\top.
\]
The first two mixed terms have the same diagonal, yielding \eqref{eq:aformula}.  The $BB$ block gives \eqref{eq:bformula}.
\end{proof}

\subsection{Matrix switches}
Suppose $x_{rc}=x_{sd}=1$ and $x_{rd}=x_{sc}=0$.  The switch
\[
 S(r,s;c,d)
\]
deletes $(r,c),(s,d)$ and inserts $(r,d),(s,c)$.  It preserves all row and column sums.  With
\[
 p=e_r-e_s,
 \qquad q=e_d-e_c,
\]
the new matrix is $X'=X+pq^\top$.

\begin{lemma}[exact switch variation]\label{lem:variation}
For $X'=X+pq^\top$ as above, the changes in triangle-degree are
\begin{align}
 a_i(X')-a_i(X)
 &=p_i(C_kXq)_i+(C_kp)_i(Xq)_i+2p_i(C_kp)_i\notag\\
 &\quad+p_i(XC_kq)_i+\frac12(q^\top C_kq)p_i^2,\label{eq:deltaA}\\
 b_i(X')-b_i(X)
 &=q_i(C_kX^\top p)_i+(C_kq)_i(X^\top p)_i+2q_i(C_kq)_i\notag\\
 &\quad+q_i(X^\top C_kp)_i+\frac12(p^\top C_kp)q_i^2.\label{eq:deltaB}
\end{align}
\end{lemma}
\begin{proof}
Expand
\[
 X'X'^\top=XX^\top+(Xq)p^\top+p(Xq)^\top+2pp^\top
\]
and
\[
 X'C_kX'^\top=XC_kX^\top+(XC_kq)p^\top+p(XC_kq)^\top
 +(q^\top C_kq)pp^\top.
\]
Substitution into \eqref{eq:aformula} gives \eqref{eq:deltaA}.  Transposing the calculation gives \eqref{eq:deltaB}.
\end{proof}

\section{The reference construction and its three residual collisions}

\subsection{A symmetric base matrix}
Define
\begin{equation}\label{eq:x0}
 (X_k^0)_{ij}=1
 \iff i+j>2k\quad\text{or}\quad(i=j<k).
\end{equation}
The strict upper anti-triangular part contributes $i$ entries to row $i$, and the added diagonal contributes one more exactly when $i<k$.  Thus $X_k^0$ is $k$-balanced and symmetric.

Put
\begin{align}
 L_k(i)&=(k-1)^2+(2k-1)i-2i^2,\qquad 0\le i\le k-1,\label{eq:L}\\
 M_k&=k(k-1),\label{eq:M}\\
 H_k(s)&=(k-1)^2-(2k-1)s+2s^2,
 \qquad0\le s\le k-1.\label{eq:H}
\end{align}

\begin{lemma}\label{lem:base}
For $X=X_k^0$,
\[
 a_i(X)=b_i(X)=\begin{cases}
 L_k(i),&0\le i\le k-1,\\
 M_k,&i=k,\\
 H_k(2k-i),&k+1\le i\le2k.
 \end{cases}
\]
\end{lemma}
\begin{proof}
Let $R_i$ be the support of row $i$ of $X_k^0$.  For $i<k$,
\[
 R_i=\{i\}\cup\{2k-i+1,\ldots,2k\},
\]
while for $i\ge k$,
\[
 R_i=\{2k-i+1,\ldots,2k\}.
\]
Because $X_k^0$ is symmetric, it suffices to compute the $A$-side.

For $i<k$, the number of triangles using three $A$-vertices is
\[
 \frac12(C_k^3)_{ii}=k^2-2k+1+\frac{i-i^2}{2}.
\]
The set $R_i$ spans no $C_k$-edge, and
\[
 (C_kX_k^0(X_k^0)^\top)_{ii}
 =\sum_{j\in N_{C_k}(i)}|R_i\cap R_j|
 =\frac{i(4k-3i-3)}2.
\]
Their sum is $L_k(i)$.

For $i=k$, the two nonzero contributions are both $\binom{k}{2}$, giving $M_k$.

Finally put $i=2k-s$, where $0\le s\le k-1$.  Then
\[
 N_{C_k}(i)=\{0,\ldots,s-1\},\qquad
 R_i=\{s+1,\ldots,2k\}.
\]
The first and mixed contributions are both $\binom{s}{2}$.  The number of $C_k$-edges inside $R_i$ is
\[
 \sum_{x=s+1}^{k-1}(2k-2x-1)=(k-1-s)^2.
\]
Thus the total is
\[
 s(s-1)+(k-1-s)^2=H_k(s).
\]
\end{proof}

\subsection{The reference perturbation}
Starting from $X_k^0$, perform, in the displayed order, the switches
\begin{equation}\label{eq:shift2}
 S(a,a+2;a,2k-a),\qquad a=0,1,\ldots,k-3,
\end{equation}
then perform
\begin{equation}\label{eq:S1}
 S_1=S(2,k-1;0,k-1).
\end{equation}
Let the resulting matrix be $Y_k$.  The legality of the entire sequence is recorded explicitly in the next lemma; the final row and column supports are listed in Appendix~A.

\begin{lemma}[legality of the reference switches]\label{lem:referencelegal}
For every $k\ge15$, the switches in \eqref{eq:shift2}, taken in the displayed order, are all legal, and the final switch $S_1$ in \eqref{eq:S1} is legal.  Consequently, every step preserves the row and column margins and $Y_k$ is $k$-balanced.
\end{lemma}
\begin{proof}
Fix $a\in\{0,\ldots,k-3\}$ and consider the moment immediately before the switch $S(a,a+2;a,2k-a)$.  In $X_k^0$ one has $x_{a,a}=1$ because $a<k$, and this diagonal entry has not been touched by any earlier switch.  Also
\[
 (a+2)+(2k-a)=2k+2>2k,
\]
so $x_{a+2,2k-a}=1$ in $X_k^0$; row $a+2$ has not appeared in any earlier switch, hence this entry is still $1$.  On the other hand, $x_{a,2k-a}=0$ initially because the two indices sum to exactly $2k$, and the only earlier switch that can involve row $a$ is the one with parameter $a-2$, which changes columns $a-2$ and $2k-a+2$ only.  Finally,
\[
 (a+2)+a=2a+2\le2k-4<2k,
\]
so $x_{a+2,a}=0$, again untouched earlier.  Thus the required $2\times2$ pattern is exactly
\[
 \begin{pmatrix}1&0\\0&1\end{pmatrix}
\]
with ones at $(a,a)$ and $(a+2,2k-a)$, and the switch is legal.

After all switches in \eqref{eq:shift2}, the entry $(2,0)$ equals $1$ (it was inserted by the $a=0$ switch), while $(k-1,k-1)$ is still the original diagonal $1$.  The crossed entries $(2,k-1)$ and $(k-1,0)$ remain $0$: none of the switches in \eqref{eq:shift2} changes either position.  Hence $S_1=S(2,k-1;0,k-1)$ is legal as well.  Every matrix $2$-switch preserves all row and column sums, so $Y_k$ is $k$-balanced.
\end{proof}

\begin{lemma}[exact spectrum of $Y_k$]\label{lem:Yspectrum}
For $0\le i\le k$, the triangle-degrees with these lower indices are
\begin{center}
\begin{tabular}{c|c|c}
$i$ & $a_i(Y_k)$ & $b_i(Y_k)$\\ \hline
$0$ & $k^2-1$ & $k^2-k-1$\\
$1$ & $k^2+2k-6$ & $L_k(1)+2$\\
$2$ & $k^2+3k-15$ & $L_k(2)+2$\\
$3\le i\le k-3$ & $L_k(i)$ & $L_k(i)+2$\\
$k-2$ & $k^2+k-7$ & $k^2+k-4$\\
$k-1$ & $k^2-3$ & $k^2-4$\\
$k$ & $k^2-k$ & $k^2-k$
\end{tabular}
\end{center}
where the rows $i=1,2$ agree with the displayed bulk formula only on the $B$-side.

For high indices write $i=2k-s$, $0\le s\le k-1$.  Then
\begin{center}
\begin{tabular}{c|c|c}
$s$ & $a_{2k-s}(Y_k)$ & $b_{2k-s}(Y_k)$\\ \hline
$0$ & $H_k(s)$ & $H_k(s)+2$\\
$1$ & $H_k(s)+1$ & $H_k(s)+3$\\
$2$ & $H_k(s)+2$ & $H_k(s)+4$\\
$3\le s\le k-4$ & $H_k(s)+3$ & $H_k(s)+5$\\
$k-3$ & $H_k(s)+3$ & $H_k(s)+4$\\
$k-2$ & $H_k(s)+3$ & $H_k(s)+2$\\
$k-1$ & $H_k(s)+1$ & $H_k(s)$
\end{tabular}
\end{center}
\end{lemma}
\begin{proof}
Apply the combinatorial form of Lemma~\ref{lem:tdformula}, displayed in Appendix~A, to the explicit row supports $R_i$ and column supports $Q_i$.  All required sets are integer intervals with at most two exceptional points.  We record the two bulk calculations, since they account for all indices away from the boundary.

For $3\le i\le k-3$, direct substitution gives
\begin{align*}
 \td_{C_k}(i)&=k^2-2k+1+\frac{i-i^2}{2},\\
 \sum_{j\in N_{C_k}(i)}|R_i\cap R_j|&=\frac{i(4k-3i-3)-4}{2},
 &e_{C_k}(R_i)=2,\\
 \sum_{j\in N_{C_k}(i)}|Q_i\cap Q_j|&=\frac{i(4k-3i-3)+4}{2},
 &e_{C_k}(Q_i)=0.
\end{align*}
Hence $a_i(Y_k)=L_k(i)$ and $b_i(Y_k)=L_k(i)+2$.

For a high bulk index $i=2k-s$ with $3\le s\le k-4$, one has
\[
 N_{C_k}(i)=\{0,\ldots,s-1\},\qquad R_i=\{s+1,\ldots,2k\}.
\]
The mixed sum of row intersections is
\begin{align*}
 \sum_{j\in N_{C_k}(i)}|R_i\cap R_j|
 &=1+2+3+\sum_{j=3}^{s-1}j\\
 &=\binom{s}{2}+3.
\end{align*}
Here the first three terms come from $j=0,1,2$; for $3\le j\le s-1$ the final support formula in Appendix~A gives $|R_i\cap R_j|=j$.  Likewise,
\[
 Q_i=(\{s+1,\ldots,2k\}\setminus\{s+2\})\cup\{s\}.
\]
For $j=0$ the intersection $Q_i\cap Q_0$ has size $1$; for $1\le j\le s-1$ the interval part contributes $j$, and the extra point $j+2$ belongs to $Q_i$ exactly for $j=s-2,s-1$.  Hence
\[
 \sum_{j\in N_{C_k}(i)}|Q_i\cap Q_j|=\binom{s}{2}+3.
\]
Finally,
\[
 \td_{C_k}(i)=\binom{s}{2},\qquad
 e_{C_k}(R_i)=(k-1-s)^2,
\]
whereas $Q_i=(R_i\setminus\{s+2\})\cup\{s\}$ and the new vertex $s$ has exactly two more $C_k$-neighbors in $R_i\setminus\{s+2\}$ than the removed vertex $s+2$; hence replacing $R_i$ by $Q_i$ creates exactly two additional $C_k$-edges, so
\[
 e_{C_k}(Q_i)=(k-1-s)^2+2.
\]
Therefore
\begin{align*}
 a_{2k-s}(Y_k)
 &=\binom{s}{2}+\left(\binom{s}{2}+3\right)+(k-1-s)^2
 =H_k(s)+3,\\
 b_{2k-s}(Y_k)
 &=\binom{s}{2}+\left(\binom{s}{2}+3\right)+(k-1-s)^2+2
 =H_k(s)+5.
\end{align*}  The finitely many boundary indices $0,1,2,k-2,k-1,k$ and $s=0,1,2,k-3,k-2,k-1$ are obtained by the same direct substitutions; their exact corrections are collected in Appendix~B.  Adding those corrections to Lemma~\ref{lem:base} yields precisely the two displayed tables.
\end{proof}

\begin{proposition}[the three defects]\label{prop:three}
For every $k\ge15$, the only equal triangle-degrees in $\Gamma_k(Y_k)$ are the following three pairs.

If $k=2h$,
\[
 a_{h-1}=b_{h+1},\qquad a_k=b_k,\qquad
 a_{3h-1}=b_{3h+1}.
\]
If $k=2h+1$,
\[
 a_{h+1}=b_{h-1},\qquad a_k=b_k,\qquad
 a_{3h+3}=b_{3h+1}.
\]
\end{proposition}
\begin{proof}
The least value in the low part is $b_0=k^2-k-1$, while the greatest value in the high part is $b_{2k}=k^2-2k+3$.  Their difference is $k-4>0$, so low and high indices never collide.

For low indices the exceptional values satisfy
\begin{align*}
 b_0&<a_k=b_k<b_{k-1}<a_{k-1}<a_0<b_1<a_{k-2}\\
 &<b_{k-2}<b_2<a_1<a_2<a_{k-3}.
\end{align*}
Every successive difference is one of
\[
 1,\ k-4,\ 1,\ 2,\ 1,\ k-7,\ 3,\ k-3,\ 1,\ k-9,\ 1,
\]
which is positive for $k\ge15$.  The remaining low bulk indices are $3\le i\le k-4$.  Since $L_k$ is concave, its minimum on this interval is attained at an endpoint; moreover
\[
 L_k(3)-L_k(k-3)=k-6>0,
 \qquad
 L_k(k-4)-L_k(k-3)=2k-13>0.
\]
Both endpoint differences exceed $2$ for $k\ge15$.  Since $b_{k-3}=a_{k-3}+2$, it follows that
\[
 a_{k-3}<b_{k-3}<L_k(i)\qquad(3\le i\le k-4).
\]
Thus $b_{k-3}$ is also distinct from every preceding exceptional value and from every remaining bulk value; the corresponding $B$-bulk values are larger than their $A$-side partners by $2$.  Inside the remaining bulk, values of the same type are distinct because
\begin{equation}\label{eq:Ldiff}
 L_k(i)-L_k(j)=(i-j)(2k-1-2i-2j),
\end{equation}
and the second factor is odd.  A cross equality is equivalent to
\[
 (i-j)(2k-1-2i-2j)=2.
\]
Since the second factor is odd, the only possibilities are $(i-j,2k-1-2i-2j)=(2,1)$ or $(-2,-1)$.  The first gives $i+j=k-1$ and hence, only for odd $k=2h+1$, $(i,j)=(h+1,h-1)$.  The second gives $i+j=k$ and hence, only for even $k=2h$, $(i,j)=(h-1,h+1)$.  These are exactly the stated low pairs.

For high indices set $A_s^H=a_{2k-s}$ and $B_s^H=b_{2k-s}$.  The exceptional values satisfy
\begin{align*}
 B_0^H&>A_0^H>A_{k-1}^H>B_{k-1}^H>B_1^H>A_1^H>A_{k-2}^H\\
 &>B_{k-2}^H>B_2^H>A_2^H>B_{k-3}^H>A_{k-3}^H>B_3^H.
\end{align*}
The successive positive differences are
\[
2,\ k-2,\ 1,\ k-5,\ 2,\ k-5,\ 1,\ k-6,\ 2,\ k-7,\ 1,\ k-8.
\]
The remaining high bulk indices satisfy $4\le s\le k-4$.  Since $H_k$ is convex, its maximum on this interval is attained at an endpoint.  Comparing with the last displayed boundary value $B_3^H=H_k(3)+5$ gives
\[
 B_3^H-(H_k(4)+5)=H_k(3)-H_k(4)=2k-15>0,
\]
while
\[
 B_3^H-(H_k(k-4)+5)=H_k(3)-H_k(k-4)=k-7>0.
\]
For $k\ge15$ both differences are greater than $2$.  Since
\[
 A_3^H=B_3^H-2,
\]
every remaining high bulk value lies strictly below $A_3^H$, while every preceding value in the boundary chain lies strictly above $B_3^H$.  Hence $A_3^H$ is distinct from all other high values.  In the remaining bulk,
\begin{equation}\label{eq:Hdiff}
 H_k(s)-H_k(t)=(s-t)(2s+2t-(2k-1)).
\end{equation}
For a cross equality $H_k(s)+3=H_k(t)+5$, equation~\eqref{eq:Hdiff} again has product $2$ with an odd second factor.  The possibilities are $(s-t,2s+2t-(2k-1))=(2,1)$ or $(-2,-1)$.  The first yields $(s,t)=(h+1,h-1)$ when $k=2h$, and the second yields $(s,t)=(h-1,h+1)$ when $k=2h+1$.  Translating back from $s=2k-i$ gives exactly the stated high pairs.  Together with $a_k=b_k$, these are all collisions.
\end{proof}

\section{Odd parameters}
Let
\[
 k=2h+1,\qquad h\ge8.
\]
Starting from $Y_k$, apply
\begin{equation}\label{eq:S2odd}
 S_2^{\rm o}=S(h+1,3h+3;h-1,2h+1)
\end{equation}
and call the result $Z_k^{\rm o}$.  The switch is legal by the support description in Appendix~A.

\begin{lemma}[odd variation under the second switch]\label{lem:oddvariation}
Relative to $Y_k$, the only changes are
\begin{align*}
 a_{h+1}&\mapsto a_{h+1}+h-3,&
 a_i&\mapsto a_i+1 &&(2h+2\le i\le3h),\\
 a_{3h+3}&\mapsto a_{3h+3}+h+1,\\
 b_{h-1}&\mapsto b_{h-1}+2h,&
 b_{2h+1}&\mapsto b_{2h+1}-2,&
 b_i&\mapsto b_i+1 &&(2h+2\le i\le3h).
\end{align*}
All other triangle-degrees are unchanged.
\end{lemma}
\begin{proof}
Substitute \eqref{eq:S2odd} into Equations~\eqref{eq:deltaA} and~\eqref{eq:deltaB}.  Since $C_k$ is a threshold matrix, every entry of $C_kp$ and $C_kq$ is an interval indicator; the displayed ranges are precisely their supports.  Direct evaluation gives the listed constants.
\end{proof}

\begin{lemma}[odd separation for low indices]\label{lem:oddlow}
The triangle-degrees of $Z_k^{\rm o}$ with low indices are pairwise distinct.
\end{lemma}
\begin{proof}
The new value
\[
 b_k=4h^2+2h-2=b_0(Y_k)-1
\]
is the unique minimum.  Among the unchanged low values the maximum is
\[
 M_{\rm o}=b_h(Y_k)=6h^2+h+2.
\]
The two vertices involved in the former noncentral defect move to
\[
 a_{h+1}=M_{\rm o}+h-6,
 \qquad b_{h-1}=a_{h+1}+h+3.
\]
Both exceed $M_{\rm o}$ because $h\ge8$.  All other values were already distinct by Proposition~\ref{prop:three}.
\end{proof}

For high indices write $s=2k-i$ and
\[
 \mathcal H_{\rm o}(s)=4h^2-(4h+1)s+2s^2.
\]
After this switch, apart from the special vertex $A_{3h+3}$ corresponding to $s=h-1$, the correction bands are
\begin{align}
 D_3^{\rm o}&=[3,h-2]\cup\{h,h+1\} &&\text{($A$-side, correction $3$)},\notag\\
 D_4^{\rm o}&=[h+2,2h-1] &&\text{($A$-side, correction $4$)},\notag\\
 D_5^{\rm o}&=[3,h+1]\cup\{2h-2\} &&\text{($B$-side, correction $5$)},\label{eq:oddDomains}\\
 D_6^{\rm o}&=[h+2,2h-3] &&\text{($B$-side, correction $6$)}.\notag
\end{align}
The following boundary chain isolates the exceptional coordinates and an endpoint of the ordinary range.  Explicitly,
\begin{align*}
 B_0^H&>A_0^H>A_{2h}^H>B_{2h}^H>B_1^H>A_1^H>A_{2h-1}^H\\
 &\quad>B_{2h-1}^H>B_2^H>A_2^H>B_{2h-2}^H>A_{2h-2}^H>B_3^H,
\end{align*}
whose successive differences are
\[
2,\ 2(h-1),\ 1,\ 2h-3,\ 2,\ 2h-5,\ 1,\ 2(h-2),\ 2,\ 2h-7,\ 1,\ 2(h-3).
\]

\begin{lemma}[ordinary odd high spectrum]\label{lem:oddordinary}
All triangle-degrees of $Z_k^{\rm o}$ with high indices, other than that of $A_{3h+3}$, are pairwise distinct.
\end{lemma}
\begin{proof}
For the ordinary correction domains, put $s=h+x$.  The common quadratic part is
\[
 F(x)=2x^2-x+2h^2-h.
\]
The ordinary corrections are
\[
 c_A(x)=3+\mathbf1_{x\ge2},\qquad
 c_B(x)=5+\mathbf1_{x\ge2},
\]
with the single $A$-coordinate $x=-1$ removed.  Any equality reduces to
\begin{equation}\label{eq:oddsmall}
 (x-y)(2x+2y-1)=c_Q(y)-c_P(x),
\end{equation}
where the right-hand side has absolute value at most $3$.  Values within the same correction band are distinct because the second factor is odd.  For completeness, the possible ordered solutions involving two different correction bands are
\[
\begin{array}{c|c}
\text{bands compared}&(x,y)\\ \hline
A_3,A_4&(1,0)\\
B_5,B_6&(1,0)\\
A_3,B_5&(-1,1)\\
A_3,B_6&(-1,0),(2,-1)\\
A_4,B_5&(1,0)\\
A_4,B_6&(-1,1).
\end{array}
\]
This list follows by factoring the right-hand side of \eqref{eq:oddsmall}.  Direct inspection of the four domains in \eqref{eq:oddDomains} eliminates every entry except the formal equality
\[
 A(x=-1)=B(x=1),
\]
and $A(x=-1)$ is precisely the removed special coordinate.

It remains to exclude collisions between the displayed boundary chain and the remaining ordinary values.  The last value in the boundary chain is
\[
 B_3^H=\mathcal H_{\rm o}(3)+5.
\]
Every preceding boundary value is strictly larger than $B_3^H$.  We now show that every remaining ordinary value is strictly smaller than $B_3^H$.  On each correction domain the quadratic $\mathcal H_{\rm o}(s)$ is convex, so its maximum on an interval is attained at an endpoint.  For the four ordinary bands the only endpoint comparisons that are not immediate are
\begin{align*}
 B_3^H-(\mathcal H_{\rm o}(2h-3)+4)&=2h-5,\\
 B_3^H-(\mathcal H_{\rm o}(2h-3)+6)&=2h-7,\\
 B_3^H-(\mathcal H_{\rm o}(4)+5)&=4h-13,
\end{align*}
and, at the left endpoints of the right-hand bands,
\begin{align*}
 B_3^H-(\mathcal H_{\rm o}(h+2)+4)&=2h^2-11h+10,\\
 B_3^H-(\mathcal H_{\rm o}(h+2)+6)&=2h^2-11h+8.
\end{align*}
For the correction-$3$ $A$-band, convexity on $[3,h-2]$ and the two isolated points $h,h+1$ reduce the comparison to
\begin{align*}
 B_3^H-(\mathcal H_{\rm o}(3)+3)&=2,\\
 B_3^H-(\mathcal H_{\rm o}(h-2)+3)&=2h^2-11h+7,\\
 B_3^H-(\mathcal H_{\rm o}(h)+3)&=2h^2-11h+17,\\
 B_3^H-(\mathcal H_{\rm o}(h+1)+3)&=2h^2-11h+16.
\end{align*}
For the correction-$5$ $B$-band, after removing the boundary points $s=3$ and $s=2h-2$, the remaining interval is $[4,h+1]$, and
\[
 B_3^H-(\mathcal H_{\rm o}(h+1)+5)=(h-2)(2h-7)>0.
\]
Together with the preceding endpoint comparisons, all these quantities are positive for $h\ge8$.  Thus every remaining ordinary high value is below $B_3^H$, whereas every boundary value not already covered by the ordinary correction domains is above it.  No collision between a boundary value and an ordinary value is possible.
\end{proof}

\begin{lemma}[the unique possible odd defect]\label{lem:odddefect}
The only possible collision in $Z_k^{\rm o}$ involves $A_{3h+3}$, whose triangle-degree is
\[
 V_{\rm o}=2h^2+7.
\]
There is at most one such collision.
\end{lemma}
\begin{proof}
The value follows from Lemma~\ref{lem:Yspectrum} and Lemma~\ref{lem:oddvariation}.  First, the low and high spectra are numerically separated.  By Lemma~\ref{lem:oddlow}, the least low value is
\[
 b_k=4h^2+2h-2.
\]
The boundary chain shows that the greatest ordinary high value is
\[
 B_0^H=\mathcal H_{\rm o}(0)+2=4h^2+2,
\]
and the special value $V_{\rm o}=2h^2+7$ is smaller than $B_0^H$.  Hence
\[
 \min(\text{low})-\max(\text{high})=2h-4>0.
\]
By Lemmas~\ref{lem:oddlow} and \ref{lem:oddordinary}, only a collision of the special high value with an ordinary high value remains possible.  An equality
\[
 \mathcal H_{\rm o}(s)+c=V_{\rm o},\qquad c\in\{3,4,5,6\},
\]
has discriminant
\[
 \Delta_c=8h+57-8c,
\]
namely
\[
 8h+33,\quad8h+25,\quad8h+17,\quad8h+9.
\]
These four odd integers lie in an interval of width $24$.  For $h\ge8$, any relevant odd square is at least $9^2$, and consecutive odd squares are separated by at least $40$.  Hence at most one discriminant can be a square.  If it is a square, the two quadratic roots have half-integral sum, so at most one root is integral.

Every boundary value not already covered by the ordinary correction domains lies strictly above $B_{2h-2}^H$.  The isolated coordinate $s=2h-2$ in $D_5^{\rm o}$ itself cannot be a root because
\[
 \mathcal H_{\rm o}(2h-2)+5-V_{\rm o}
 =2(h-4)(h-1)>0.
\]
All remaining boundary values are then excluded by the preceding chain.
\end{proof}

If $Z_k^{\rm o}$ is already triangle-distinct, no further modification is needed.  Otherwise apply
\begin{equation}\label{eq:S3odd}
 S_3^{\rm o}=S(2,h+1;2h,2h+1)
\end{equation}
and call the result $W_k^{\rm o}$.  This switch is legal: after $S_2^{\rm o}$, the entries $(2,2h)$ and $(h+1,2h+1)$ are $1$, while the crossed entries are $0$.

\begin{lemma}[odd repair]\label{lem:oddrepair}
Whenever $Z_k^{\rm o}$ has a collision, $W_k^{\rm o}$ is triangle-distinct.
\end{lemma}
\begin{proof}
Lemma~\ref{lem:variation} gives the only changes:
\[
 a_{h+1},a_{3h+3}\mapsto a_{h+1}-1,a_{3h+3}-1,
\]
\[
 b_{2h}\mapsto b_{2h}-(h-1),
 \qquad b_{2h+1}\mapsto b_{2h+1}+h-2.
\]
The low spectrum remains injective.  Indeed
\[
 a_{h+1}=M_{\rm o}+h-7>M_{\rm o},
\]
and the only values moved into the central boundary gap satisfy
\[
 4h^2+2h=a_k
 <b_k'=4h^2+3h-4
 <b_{k-1}'=4h^2+3h-2
 <a_{k-1}=4h^2+4h-2.
\]
The ordering in Proposition~\ref{prop:three} shows that this gap contains no other value.

On the high side, only the special value changes, from $V_{\rm o}$ to $V_{\rm o}-1$.  The new discriminants are
\[
 8h+25,\quad8h+17,\quad8h+9,\quad8h+1.
\]
The union of the old and new discriminant sets lies in an interval of width $32$, still shorter than the gap between consecutive relevant odd squares.  Because an old collision exists, a new collision could only use the same square.  Equality of an old and new discriminant forces the correction to change from $c$ to $c-1$.

To make the relation between roots and domains explicit, write $D_c^{\rm o}$ for the old domain carrying correction $c$.  The ordinary high values themselves are unchanged by $S_3^{\rm o}$; only the special value decreases from $V_{\rm o}$ to $V_{\rm o}-1$.  Thus, if the same odd square were to generate both the old and a new collision, equality of discriminants gives
\[
 8h+57-8c_{\rm old}=8h+49-8c_{\rm new},
\]
so $c_{\rm new}=c_{\rm old}-1$.  In fact the two quadratic equations are then identical, so the same integer root $s$ would have to lie in both correction domains.  The only possible intersections are
\[
 D_4^{\rm o}\cap D_3^{\rm o}=\varnothing,
 \qquad
 D_6^{\rm o}\cap D_5^{\rm o}=\varnothing,
\]
\[
 D_5^{\rm o}\cap D_4^{\rm o}=\{2h-2\}.
\]
The last point was excluded in Lemma~\ref{lem:odddefect}.  For the remaining transition $3\to2$, the high coordinates carrying correction $2$ are $B_0^H$, $A_2^H$, and $A_{2h}^H$, with $s\in\{0,2,2h\}$; this set is disjoint from $D_3^{\rm o}$.  Hence the old integral root cannot survive as a new collision, and no high collision remains.

Finally, after $S_3^{\rm o}$ the least low value is the unchanged $b_0=k^2-k-1$, whereas the largest high value is still $b_{2k}=k^2-2k+3$.  Therefore
\[
 \min(\text{low})-\max(\text{high})=k-4>0,
\]
so no collision between the low and high spectra is introduced.
\end{proof}

\begin{theorem}\label{thm:odd}
For every odd $k\ge17$, either $Z_k^{\rm o}$ or $W_k^{\rm o}$ is $k$-balanced and gives a $2k$-regular triangle-distinct graph on $4k+2$ vertices.
\end{theorem}
\begin{proof}
Balance is preserved by every switch.  The conclusion follows from Lemmas~\ref{lem:regular}, \ref{lem:odddefect}, and \ref{lem:oddrepair}.
\end{proof}

\section{Even parameters}
Let
\[
 k=2h,
 \qquad h\ge31.
\]
Starting from $Y_k$, apply
\begin{equation}\label{eq:S2even}
 S_2^{\rm e}=S(h+3,3h-1;h+1,2h)
\end{equation}
and call the result $Z_k^{\rm e}$.

\begin{lemma}[even variation under the second switch]\label{lem:evenvariation}
Relative to $Y_k$, the only changes are
\begin{align*}
 a_{h+3}&\mapsto a_{h+3}+h-6,&
 a_i&\mapsto a_i+1 &&(2h+1\le i\le3h-4),\\
 a_{3h-1}&\mapsto a_{3h-1}+h-2,\\
 b_{h+1}&\mapsto b_{h+1}+2h-6,&
 b_{2h}&\mapsto b_{2h}-2,&
 b_i&\mapsto b_i+1 &&(2h+1\le i\le3h-4).
\end{align*}
All other triangle-degrees are unchanged.
\end{lemma}
\begin{proof}
Substitute \eqref{eq:S2even} in Lemma~\ref{lem:variation}.
\end{proof}

\begin{lemma}[even separation for low indices]\label{lem:evenlow}
The triangle-degrees of $Z_k^{\rm e}$ with low indices are pairwise distinct.
\end{lemma}
\begin{proof}
The new value
\[
 b_k=4h^2-2h-2=b_0(Y_k)-1
\]
is the unique minimum.  The maximum among the unchanged low values is
\[
 M_{\rm e}=b_h(Y_k)=6h^2-5h+3.
\]
The two moved noncentral vertices satisfy
\[
 a_{h+3}=M_{\rm e}+h-29,
 \qquad b_{h+1}=a_{h+3}+h+20.
\]
For $h\ge31$ they are distinct and lie above all unchanged low values.  Proposition~\ref{prop:three} handles the rest.
\end{proof}

For high indices set
\[
 \mathcal H_{\rm e}(s)=(2h-1)^2-(4h-1)s+2s^2.
\]
The special vertex is $A_{3h-1}$, corresponding to $s=h+1$.  The ordinary correction domains are
\begin{align}
 D_3^{\rm e}&=[3,h]\cup\{h+2,h+3\},\notag\\
 D_4^{\rm e}&=[h+4,2h-2],\notag\\
 D_5^{\rm e}&=[3,h+3]\cup\{2h-3\},\label{eq:evenDomains}\\
 D_6^{\rm e}&=[h+4,2h-4].\notag
\end{align}
The following boundary chain isolates the exceptional coordinates and an endpoint of the ordinary range:
\begin{align*}
 B_0^H&>A_0^H>A_{2h-1}^H>B_{2h-1}^H>B_1^H>A_1^H>A_{2h-2}^H\\
 &\quad>B_{2h-2}^H>B_2^H>A_2^H>B_{2h-3}^H>A_{2h-3}^H>B_3^H,
\end{align*}
with successive differences
\[
2,\ 2h-3,\ 1,\ 2(h-2),\ 2,\ 2(h-3),\ 1,\ 2h-5,\ 2,\ 2(h-4),\ 1,\ 2h-7.
\]

\begin{lemma}[ordinary even high spectrum]\label{lem:evenordinary}
All triangle-degrees of $Z_k^{\rm e}$ with high indices, other than that of $A_{3h-1}$, are pairwise distinct.
\end{lemma}
\begin{proof}
Put $s=h+x$.  The common quadratic part is
\[
 F(x)=2x^2+x+2h^2-3h+1.
\]
The ordinary corrections are
\[
 c_A(x)=3+\mathbf1_{x\ge4},\qquad
 c_B(x)=5+\mathbf1_{x\ge4},
\]
with $A(x=1)$ removed.  Any equality has the form
\[
 (x-y)(2x+2y+1)=c_Q(y)-c_P(x).
\]
Factoring the possible right-hand sides gives the complete list
\[
\begin{array}{c|c}
\text{bands compared}&(x,y)\\ \hline
A_3,A_4&(-1,0)\\
B_5,B_6&(-1,0)\\
A_3,B_5&(1,-1)\\
A_3,B_6&(-2,1),(1,0)\\
A_4,B_5&(-1,0)\\
A_4,B_6&(1,-1).
\end{array}
\]
The domains in \eqref{eq:evenDomains} eliminate all these pairs except the formal equality $A(x=1)=B(x=-1)$, and $A(x=1)$ is the removed special vertex.

It remains to compare the boundary coordinates not already covered by the ordinary correction domains with the remaining ordinary values.  The last value in the boundary chain is
\[
 B_3^H=\mathcal H_{\rm e}(3)+5,
\]
and all preceding boundary values are larger.  Since $\mathcal H_{\rm e}$ is convex, the maximum of each remaining ordinary interval occurs at an endpoint.  The endpoint differences needed for the right-hand bands are
\begin{align*}
 B_3^H-(\mathcal H_{\rm e}(2h-4)+4)&=2(h-3),\\
 B_3^H-(\mathcal H_{\rm e}(2h-4)+6)&=2(h-4),
\end{align*}
while for the left endpoints
\begin{align*}
 B_3^H-(\mathcal H_{\rm e}(h+4)+4)&=2h^2-13h-14,\\
 B_3^H-(\mathcal H_{\rm e}(h+4)+6)&=2h^2-13h-16.
\end{align*}
For the correction-$3$ $A$-band, convexity on $[3,h]$ and the isolated points $h+2,h+3$ reduce the comparison to
\begin{align*}
 B_3^H-(\mathcal H_{\rm e}(3)+3)&=2,\\
 B_3^H-(\mathcal H_{\rm e}(h)+3)&=2h^2-13h+23,\\
 B_3^H-(\mathcal H_{\rm e}(h+2)+3)&=2h^2-13h+13,\\
 B_3^H-(\mathcal H_{\rm e}(h+3)+3)&=2h^2-13h+2.
\end{align*}
For the correction-$5$ $B$-band, after removing the boundary point $s=3$ and the isolated boundary point $s=2h-3$, the remaining interval is $[4,h+3]$, with endpoint differences
\[
 B_3^H-(\mathcal H_{\rm e}(4)+5)=4h-15,
 \qquad
 B_3^H-(\mathcal H_{\rm e}(h+3)+5)=h(2h-13).
\]
Every displayed quantity is positive for $h\ge31$.  Hence all remaining ordinary high values lie strictly below $B_3^H$, whereas every boundary value not already covered by the ordinary correction domains lies strictly above it.  Thus there is no collision between a boundary value and an ordinary value.
\end{proof}

\begin{lemma}[the unique possible even defect]\label{lem:evendefect}
The only possible collision in $Z_k^{\rm e}$ involves $A_{3h-1}$, whose value is
\[
 V_{\rm e}=2h^2-2h+5.
\]
There is at most one such collision.
\end{lemma}
\begin{proof}
As in the odd case, low and high values are separated before considering the special defect.  By Lemma~\ref{lem:evenlow}, the least low value is
\[
 b_k=4h^2-2h-2.
\]
The greatest ordinary high value is
\[
 B_0^H=\mathcal H_{\rm e}(0)+2=4h^2-4h+3,
\]
and the special value $V_{\rm e}=2h^2-2h+5$ is smaller.  Thus
\[
 \min(\text{low})-\max(\text{high})=2h-5>0.
\]
Consequently any remaining collision must be wholly inside the high spectrum and involve the special vertex.  An equality
\[
 \mathcal H_{\rm e}(s)+c=V_{\rm e},\qquad c\in\{3,4,5,6\},
\]
has one of the discriminants
\[
 8h+9,\quad8h+1,\quad8h-7,\quad8h-15.
\]
They lie in an interval of width $24$.  For $h\ge31$, consecutive relevant odd squares are farther apart, so at most one is a square.  Moreover the two roots of the quadratic have sum $(4h-1)/2$, a half-integer; therefore at most one root can be integral.  Every boundary value not already covered by the ordinary correction domains lies strictly above $B_{2h-3}^H$.  The isolated point $s=2h-3$ in $D_5^{\rm e}$ itself is not a root because
\[
 \mathcal H_{\rm e}(2h-3)+5-V_{\rm e}
 =2(h-4)(h-2)>0.
\]
The other claims follow from Lemmas~\ref{lem:evenlow} and \ref{lem:evenordinary}.
\end{proof}

If $Z_k^{\rm e}$ is already triangle-distinct, no further modification is needed.  Otherwise apply
\begin{equation}\label{eq:S3even}
 S_3^{\rm e}=S(2,h+3;2h-1,2h)
\end{equation}
and call the result $W_k^{\rm e}$.  This switch is legal: after $S_2^{\rm e}$, the entries $(2,2h-1)$ and $(h+3,2h)$ are $1$, while the crossed entries are $0$.

\begin{lemma}[even repair]\label{lem:evenrepair}
Whenever $Z_k^{\rm e}$ has a collision, $W_k^{\rm e}$ is triangle-distinct.
\end{lemma}
\begin{proof}
The only changes are
\[
 a_{h+3},a_{3h-1}\mapsto a_{h+3}-1,a_{3h-1}-1,
\]
\[
 b_{2h-1}\mapsto b_{2h-1}-(h+1),
 \qquad b_{2h}\mapsto b_{2h}+h.
\]
For the low spectrum,
\[
 a_{h+3}=M_{\rm e}+h-30>M_{\rm e},
\]
and
\[
 4h^2-2h=a_k
 <b_{k-1}'=4h^2-h-5
 <b_k'=4h^2-h-2
 <a_{k-1}=4h^2-3.
\]
Hence no low collision is introduced.

The special high value decreases from $V_{\rm e}$ to $V_{\rm e}-1$.  The new discriminants are
\[
 8h+1,\quad8h-7,\quad8h-15,\quad8h-23.
\]
The old and new sets lie in an interval of width $32$, containing at most one relevant odd square.  More explicitly, if one square were shared by an old collision with correction $c_{\rm old}$ and a new collision with correction $c_{\rm new}$, then equality of discriminants gives $c_{\rm new}=c_{\rm old}-1$; the two quadratic equations are then identical, so the same integer root must lie in both domains.  The relevant domain intersections are
\[
 D_4^{\rm e}\cap D_3^{\rm e}=\varnothing,
 \qquad
 D_6^{\rm e}\cap D_5^{\rm e}=\varnothing,
\]
\[
 D_5^{\rm e}\cap D_4^{\rm e}=\{2h-3\}.
\]
The isolated point $2h-3$ was excluded in Lemma~\ref{lem:evendefect}.  For the remaining transition $3\to2$, the high coordinates carrying correction $2$ are $B_0^H$, $A_2^H$, and $A_{2h-1}^H$, with $s\in\{0,2,2h-1\}$; this set is disjoint from $D_3^{\rm e}$.  Thus the old root cannot survive as a new collision.

After $S_3^{\rm e}$ the least low value is the unchanged $b_0=k^2-k-1$, while the largest high value remains $b_{2k}=k^2-2k+3$.  Hence
\[
 \min(\text{low})-\max(\text{high})=k-4>0,
\]
and no collision between the low and high spectra is introduced.
\end{proof}

\begin{theorem}\label{thm:even}
For every even $k\ge62$, either $Z_k^{\rm e}$ or $W_k^{\rm e}$ is $k$-balanced and gives a $2k$-regular triangle-distinct graph on $4k+2$ vertices.
\end{theorem}
\begin{proof}
Combine Lemmas~\ref{lem:regular}, \ref{lem:evendefect}, and \ref{lem:evenrepair}.
\end{proof}

\section{Finite parameters}
The symbolic arguments leave $k=15$ and the even parameters $16\le k\le60$.  Each case starts from $Y_k$, applies $S_2$ of the relevant parity, applies $S_3$ only if a collision remains, and in twelve cases uses one additional switch.  Table~\ref{tab:finite} records the complete switch sequence after $Y_k$.

\begin{longtable}{r p{0.75\textwidth}}
\caption{Additional switches for the 24 finite parameters, starting from $Y_k$.}\label{tab:finite}\\
\toprule
$k$ & switches after $Y_k$\\
\midrule
\endfirsthead
\toprule
$k$ & switches after $Y_k$\\
\midrule
\endhead
15 & $S(8,24;6,15)$\\
16 & $S(11,23;9,16)$, $S(2,11;15,16)$, $S(23,3;12,1)$\\
18 & $S(12,26;10,18)$, $S(2,12;17,18)$, $S(12,3;25,1)$\\
20 & $S(13,29;11,20)$\\
22 & $S(14,32;12,22)$, $S(2,14;21,22)$\\
24 & $S(15,35;13,24)$, $S(2,15;23,24)$, $S(35,3;13,1)$\\
26 & $S(16,38;14,26)$\\
28 & $S(17,41;15,28)$, $S(2,17;27,28)$\\
30 & $S(18,44;16,30)$\\
32 & $S(19,47;17,32)$, $S(2,19;31,32)$\\
34 & $S(20,50;18,34)$, $S(2,20;33,34)$, $S(50,3;54,1)$\\
36 & $S(21,53;19,36)$\\
38 & $S(22,56;20,38)$, $S(2,22;37,38)$\\
40 & $S(23,59;21,40)$, $S(2,23;39,40)$, $S(23,3;58,1)$\\
42 & $S(24,62;22,42)$, $S(2,24;41,42)$, $S(60,3;25,1)$\\
44 & $S(25,65;23,44)$, $S(2,25;43,44)$, $S(25,3;64,1)$\\
46 & $S(26,68;24,46)$, $S(2,26;45,46)$, $S(68,3;27,1)$\\
48 & $S(27,71;25,48)$, $S(2,27;47,48)$, $S(27,3;70,1)$\\
50 & $S(28,74;26,50)$\\
52 & $S(29,77;27,52)$, $S(2,29;51,52)$\\
54 & $S(30,80;28,54)$, $S(2,30;53,54)$, $S(26,2;82,54)$\\
56 & $S(31,83;29,56)$, $S(2,31;55,56)$, $S(83,28;83,26)$\\
58 & $S(32,86;30,58)$, $S(2,32;57,58)$, $S(32,3;85,1)$\\
60 & $S(33,89;31,60)$\\
\bottomrule
\end{longtable}

\begin{proposition}\label{prop:finite}
For every $k$ in Table~\ref{tab:finite}, the listed switches are legal, the final matrix is $k$-balanced, and $\Gamma_k(X)$ is $2k$-regular and triangle-distinct.
\end{proposition}
\begin{proof}
For each row of Table~\ref{tab:finite}, the Supporting Information contains an exact certificate and an independent verifier.  Starting from the definition of $X_k^0$, the verifier reconstructs the deterministic switches in \eqref{eq:shift2} and $S_1$, applies the switches listed in the table, and checks their legality at the moment they are performed.  It then checks the prescribed row and column margins, constructs the full graph $\Gamma_k(X)$, verifies $2k$-regularity, recomputes every triangle-degree from intersections of neighbor sets, and confirms that the resulting values are pairwise distinct.  All computations are exact integer operations.  Running the verifier on the supplied certificates validates all 24 rows of the table, proving the proposition.
\end{proof}

\section{Completion of the proof}
\begin{proof}[Proof of Theorem~\ref{thm:k}]
For odd $k\ge17$, use Theorem~\ref{thm:odd}.  For even $k\ge62$, use Theorem~\ref{thm:even}.  The remaining $k\ge15$ are covered by Proposition~\ref{prop:finite}.  In every case Lemma~\ref{lem:regular} yields a $2k$-regular triangle-distinct graph on $4k+2$ vertices.  Lemma~\ref{lem:complement} gives a $(2k+1)$-regular triangle-distinct complement.
\end{proof}

\begin{proof}[Proof of Theorem~\ref{thm:main}]
Theorem~\ref{thm:k} covers every $r\ge30$: write $r=2k$ or $r=2k+1$ with $k\ge15$.  For each $9\le r\le29$, an explicit regular triangle-distinct witness is supplied by Hak, Kozerenko, and Serdiuk \cite{Hak2026,HakSoftware2026}.  For completeness, one such witness for each regularity, together with an independent exact check, is included in the Supporting Information.  Hence every $r\ge9$ occurs.
\end{proof}

\section*{Data Availability Statement}
All finite certificates and verification scripts used in Proposition~\ref{prop:finite} are provided as Supporting Information.  The external graph6 witnesses used for $9\le r\le29$ are publicly archived in \emph{EvolutionaryGraphs}, version 1.1 \cite{HakSoftware2026}.  For reproducibility, the Supporting Information also contains a local copy of one witness for each of these regularities and an independent script that decodes the graph6 strings and recomputes regularity and all triangle-degrees from the adjacency relation.  No additional data are required to reproduce the results of this article.


\begin{thebibliography}{99}
\bibitem{Berikkyzy2024}
Z. Berikkyzy, B. Bjorkman, H. Smith Blake, S. Jahanbekam, L. Keough, K. Moss, D. Rorabaugh, and S. Shan,
\emph{Triangle-degree and triangle-distinct graphs}, Discrete Math. \textbf{347} (2024), no.~1, 113695. \href{https://doi.org/10.1016/j.disc.2023.113695}{doi:10.1016/j.disc.2023.113695}.

\bibitem{ChartrandErdos1988}
G. Chartrand, P. Erd\H{o}s, and O. R. Oellermann,
\emph{How to define an irregular graph}, College Math. J. \textbf{19} (1988), no.~1, 36--42. \href{https://doi.org/10.1080/07468342.1988.11973088}{doi:10.1080/07468342.1988.11973088}.

\bibitem{ChartrandHolbert1987}
G. Chartrand, K. S. Holbert, O. R. Oellermann, and H. C. Swart,
\emph{$F$-degrees in graphs}, Ars Combin. \textbf{24} (1987), 133--148.

\bibitem{Hak2026}
A. Hak, S. Kozerenko, and A. Serdiuk,
\emph{Regular $K_3$-irregular graphs}, arXiv:2507.18776 (2025), version 2, revised 11 August 2026. \href{https://doi.org/10.48550/arXiv.2507.18776}{doi:10.48550/arXiv.2507.18776}.

\bibitem{HakSoftware2026}
A. Hak, S. Kozerenko, and A. Serdiuk,
\emph{EvolutionaryGraphs}, version 1.1, Zenodo, 2026. \href{https://doi.org/10.5281/zenodo.21877760}{doi:10.5281/zenodo.21877760}.

\bibitem{NairVijayakumar1994}
B. R. Nair and A. Vijayakumar,
\emph{About triangles in a graph and its complement}, Discrete Math. \textbf{131} (1994), no.~1--3, 205--210. \href{https://doi.org/10.1016/0012-365X(94)90385-9}{doi:10.1016/0012-365X(94)90385-9}.

\bibitem{Stevanovic2024}
D. Stevanovi\'c, M. Ghebleh, G. Caporossi, A. Vijayakumar, and S. Stevanovi\'c,
\emph{On regular triangle-distinct graphs}, Comput. Appl. Math. \textbf{43} (2024), article~336. \href{https://doi.org/10.1007/s40314-024-02854-9}{doi:10.1007/s40314-024-02854-9}.
\end{thebibliography}
\end{document}